\documentclass{amsart}

\usepackage{amsmath,amssymb,amsthm,amsfonts}
\usepackage[alphabetic]{amsrefs}
\usepackage{bm}
\usepackage{mathtools}
\usepackage{hyperref}

\usepackage[margin=1in]{geometry}

\theoremstyle{definition}
\newtheorem{thm}{Theorem}[section]

\newtheorem{lem}[thm]{Lemma}

\newtheorem{defi}[thm]{Definition}

\theoremstyle{remark}
\newtheorem{rem}[thm]{Remark}

\newtheorem*{acknowledgment}{Acknowledgments}

\numberwithin{equation}{section}
\numberwithin{figure}{section}

\DeclareMathOperator{\Id}{Id}
\DeclareMathOperator{\op}{op}

\def\co{\colon\thinspace}
\def\ot{\otimes}

\newcommand{\HD}{\mathcal{H}}
\newcommand{\cop}{\operatorname{cop}}

\newcommand{\eqbyref}[1]{\overset{\mathclap{\eqref{#1}\;}}{=\;}}

\title{A Quasi-Pentagon Equation for a Heisenberg Double of a Quasi-Hopf Algebra}

\author{Yohei Ota}
\address{Department of Mathematical and Computing Science, School of Computing,
	Institute of Science Tokyo,
	2-12-1, Ookayama, Meguro-ku, Tokyo 152-8552, Japan}
\email{ota.y.aj@m.titech.ac.jp}

\begin{document}
	\begin{abstract}
		For a finite-dimensional Hopf algebra $H$, the canonical elements of the Heisenberg doubles $\HD(H^\ast)$ and $\HD(H)$ satisfy the pentagon and Hopf equations, respectively. In this paper we construct quasi-Hopf analogues of these structures. For a finite-dimensional quasi-Hopf algebra $H$, we consider natural quasi-Hopf analogues $\HD_1(H^\ast)$ and $\HD_1(H)$ of $\HD(H^\ast)$ and $\HD(H)$. Although their canonical elements are defined just as in the Hopf algebra case, they need not be invertible. We prove that there nevertheless exist natural inverse-like elements. In $\HD_1(H^\ast)$, the canonical element satisfies a quasi-pentagon equation and its inverse-like element satisfies a quasi-Hopf equation, while in $\HD_1(H)$ the roles are reversed.
	\end{abstract}
	
	\maketitle
	
	\section{Introduction}
	
	For a finite-dimensional Hopf algebra $H$, the Heisenberg double
	$\HD(H^\ast)$
	of the dual Hopf algebra $H^\ast$ carries a canonical element $W$
	satisfying the pentagon equation (see \cite{Ka}\footnote{In Kashaev's notation \cite{Ka}, the algebra $\HD(H^\ast)$ is denoted by
		$H(H)$, and is called the Heisenberg double of $H$.}):
	\begin{equation*}
		W_{12}W_{13}W_{23}=W_{23}W_{12}.
	\end{equation*}
	Since $W$ is invertible, its inverse $\tilde{W}$ satisfies
	\begin{equation*}
		\tilde{W}_{23}\tilde{W}_{13}\tilde{W}_{12}=\tilde{W}_{12}\tilde{W}_{23}.
	\end{equation*}
	Following Militaru's terminology \cite{Mi}, we call an identity of this form the
	\emph{Hopf equation}. By \cite[Proposition~5.3]{Lu}, $\HD(H)$ is isomorphic to the opposite algebra $\HD(H^\ast)^{\op}$. It follows that the canonical element
	of $\HD(H)$ satisfies the Hopf equation, while its inverse satisfies the pentagon
	equation.
	
	These identities are useful for several reasons. First, Kashaev showed that the Drinfel'd double can be realized as a subalgebra of $\HD(H^\ast)\ot \HD(H)$, and that its universal $R$-matrix admits a factorized expression as a product of four canonical elements in this Heisenberg-double setting, two coming from $\HD(H^\ast)$ and $\HD(H)$ and two of mixed type. As a consequence, the Yang--Baxter equation
	can be derived from repeated use of the pentagon and Hopf equations \cite{Ka}.
	Second, pentagon relations arising from Heisenberg doubles also have topological
	applications. For instance, Mihalache, Suzuki, and Terashima constructed an invariant
	of closed oriented $3$-manifolds from the Heisenberg double of a finite-dimensional
	involutory Hopf algebra; in their construction, the pentagon equation is the local
	algebraic relation underlying invariance under the relevant moves \cites{MST,Su}.
	
	When $H$ is a quasi-Hopf algebra, several inequivalent generalizations of the Heisenberg
	double are known; see \cite{PO}. In \cites{BC,PO}, the first Heisenberg double
	$\HD_1(H)$ was introduced. In the present paper, we also need a quasi-Hopf analogue of
	the Heisenberg double $\HD(H^\ast)$ of the dual Hopf algebra $H^\ast$. Although the
	definition of this object is naturally suggested by that of $\HD_1(H)$, we define it
	explicitly in Section~3 and denote it by $\HD_1(H^\ast)$. In the Hopf algebra case, once
	$\HD(H)$ is defined, $\HD(H^\ast)$ is obtained by the same construction, since $H^\ast$
	is again a finite-dimensional Hopf algebra. For a finite-dimensional quasi-Hopf algebra
	$H$, however, the linear dual $H^\ast$ is in general not a quasi-Hopf algebra; instead,
	it carries the structure of a dual quasi-Hopf algebra. Accordingly, in the quasi-Hopf
	setting $\HD_1(H^\ast)$ and $\HD_1(H)$ must be defined separately rather than by applying to $H^\ast$ the same formal construction that defines $\HD_1(H)$. Here $\HD_1(H^\ast)$ is an algebra
	object in the monoidal category $\mathcal M_H$ of right $H$-modules, while
	$\HD_1(H)$ is an algebra object in the monoidal category ${}_H\mathcal M$ of left
	$H$-modules. (For comparison, the \emph{second} Heisenberg double $\HD_2(H)$ in
	\cite{PO} is an algebra object in $\mathcal M_H$, while the \emph{third} Heisenberg
	double $\HD_3(H)$ considered by Panaite \cite{Pa} is an algebra object in the bimodule
	category ${}_{D(H)}\mathcal M_{D(H)^{\cop}}$, where $D(H)$ is the Drinfel'd double of
	$H$.) Since $H$ is a quasi-Hopf algebra, these multiplications are in general not
	associative.
	
	In this paper, we construct quasi-Hopf analogues of the four identities appearing above.
	By quasi-Hopf analogues, we mean versions of these equations modified by correction terms arising from
	the associator $\Phi$; we call the resulting identities the \emph{quasi-pentagon
	equations} and the \emph{quasi-Hopf equations}. We emphasize that, when $H$ is a Hopf
	algebra, these four equations follow from a single relation, namely the pentagon equation
	of the Heisenberg double, together with the invertibility of the canonical element, as
	explained above. In contrast, in the quasi-Hopf setting the existence of such equations is
	not automatic and cannot be deduced from a single relation. In particular, the canonical
	elements are defined in exactly the same natural way as in the Hopf algebra case, but
	they need not be invertible (see Section~5 for an explicit example). Thus, in order to formulate the quasi-Hopf equation for $\HD_1(H^\ast)$ and the
	quasi-pentagon equation for $\HD_1(H)$, we construct elements playing the
	roles of inverses of the canonical elements. If
	$\{\bm{e}_i\}_{i\in I}$ is a basis of $H$ and $\{\bm{e}^i\}_{i\in I}$ its dual basis,
	then the canonical elements are
	\begin{equation*}
		W=\sum_{i \in I} \varepsilon \# \bm{e}_i \ot \bm{e}^i \# 1_H
		\in {\HD_1(H^\ast)}^{\ot 2},
		\qquad
		\bar{W}=\sum_{i \in I} \bm{e}_i \# \varepsilon \ot 1_H \# \bm{e}^i
		\in {\HD_1(H)}^{\ot 2}.
	\end{equation*}
	We show that there exist natural elements
	$\tilde{W}\in {\HD_1(H^\ast)}^{\ot 2}$ for $W$ and
	$\hat{W}\in {\HD_1(H)}^{\ot 2}$ for $\bar{W}$, which we call quasi-inverses
	of the canonical elements, and which reduce to the genuine inverses in the Hopf algebra
	case. We also show that
	\begin{equation*}
		W_{12}W_{13}W_{23}=W_{23}W_{12}\boldsymbol{\Phi}^{-1},
		\qquad
		\tilde{W}_{23}\tilde{W}_{13}\tilde{W}_{12}
		=\boldsymbol{\Phi}^{321}_{S}\tilde{W}_{12}\tilde{W}_{23},
	\end{equation*}
	in $\HD_1(H^\ast)$, and
	\begin{equation*}
		\bar{W}_{23}\bar{W}_{13}\bar{W}_{12}
		=(\overline{\boldsymbol{\Phi}}^{-1})^{321}\bar{W}_{12}\bar{W}_{23},
		\qquad
		\hat{W}_{12}\hat{W}_{13}\hat{W}_{23}
		=\hat{W}_{23}\hat{W}_{12}\overline{\boldsymbol{\Phi}}_S,
	\end{equation*}
	in $\HD_1(H)$, where
	$\boldsymbol{\Phi}^{-1},\boldsymbol{\Phi}^{321}_{S},
	(\overline{\boldsymbol{\Phi}}^{-1})^{321},\overline{\boldsymbol{\Phi}}_S$
	are correction terms determined by the associator and the antipode; see
	Theorems~\ref{thm:4.4} and \ref{thm:4.5} for their explicit formulas. We call the first
	and fourth identities the quasi-pentagon equations, and the second and third identities
	the quasi-Hopf equations. In particular, when $H$ is an ordinary Hopf algebra, all
	correction terms disappear and $\tilde{W}$ and $\hat{W}$ become the usual inverses of
	the canonical elements.
	
	A further motivation for the present paper comes from the class of quasi-Hopf algebras of the form $k^{\omega}(G)$, where $k$ is a field, $G$ is a finite group, and $\omega$ is a normalized $3$-cocycle on $G$ with values in $k^\times$. We refer to \cite{DPR} for the twisted quantum double $D^{\omega}(G)$ associated with a finite group and a $3$-cocycle, and to \cite{DW} for the related finite-group and group-cohomological constructions in topological gauge theory. In this setting, one is naturally led to consider an embedding of the twisted quantum double $D^{\omega}(G)$ into $\HD_1({k^{\omega}(G)}^{\ast})\ot \HD_1(k^{\omega}(G))$, extending Kashaev's embedding \cite{Ka} to the twisted case. A related construction for the twisted quantum double will appear elsewhere \cite{Ota}. Under such an embedding, the two sides of the quasi-Yang--Baxter equation should be related by repeated use of the quasi-pentagon and quasi-Hopf equations. This indicates that, at least for this particular class of quasi-Hopf algebras, the quasi-pentagon and quasi-Hopf equations play a role analogous to that played by the pentagon and Hopf equations in the Hopf algebra case, where they are related to Yang--Baxter-type equations through Kashaev-type embeddings. Whether similar constructions exist for more general quasi-Hopf algebras, or whether they lead to genuinely new topological constructions such as $3$-manifold invariants, remains an interesting open problem.
	
	\section{Preliminaries}
	
	Throughout this paper, $k$ is a field; all algebras, vector spaces, and maps are
	assumed to be $k$-linear. We write $\ot$ for $\ot_k$.
	
	Following Drinfel'd \cite{Dr}, we recall the definitions of quasi-bialgebras and
	quasi-Hopf algebras.
	
	\begin{defi}
		Let $H$ be a unital associative algebra, let $\Phi\in H\ot H\ot H$ be an invertible element,
		and let $\Delta\co H\to H\ot H$ and $\varepsilon\co H\to k$ be algebra maps.
		We say that $(H,\Delta,\varepsilon,\Phi)$ is a \emph{quasi-bialgebra} if for all $h\in H$,
		\begin{gather}\label{eq:2.1}
			(\Id_H \ot \Delta)(\Delta(h))=\Phi(\Delta\ot \Id_H)(\Delta(h))\Phi^{-1},\\
			\label{eq:2.2}
			(\Id_H \ot \varepsilon)(\Delta(h))=h=(\varepsilon \ot \Id_H)(\Delta(h)),\\
			\label{eq:2.3}
			(1_H \ot \Phi)(\Id_H \ot \Delta \ot \Id_H)(\Phi)(\Phi\ot 1_H)
			=(\Id_H \ot \Id_H \ot \Delta)(\Phi)(\Delta \ot \Id_H \ot \Id_H)(\Phi),\\
			\label{eq:2.4}
			(\Id_H \ot \varepsilon \ot \Id_H)(\Phi)=1_H\ot1_H.
		\end{gather}
		Using \eqref{eq:2.2}--\eqref{eq:2.4}, one can also show that
		\begin{equation*}
			(\varepsilon \ot \Id_H \ot \Id_H)(\Phi)=1_H\ot 1_H,
			\qquad
			(\Id_H \ot \Id_H \ot \varepsilon)(\Phi)=1_H\ot 1_H.
		\end{equation*}
		The element $\Phi$ is called the \emph{associator}, $\Delta$ the \emph{coproduct},
		and $\varepsilon$ the \emph{counit}.
		
		Let $(H,\Delta,\varepsilon,\Phi)$ be a quasi-bialgebra.
		Let $S\co H\to H$ be an anti-algebra map and let $\alpha,\beta\in H$.
		We say that $(H,\Delta,\varepsilon,\Phi,\alpha,\beta,S)$ is a \emph{quasi-Hopf algebra} if
		for all $h\in H$,
		\begin{alignat}{2}\label{eq:2.5}
			&\sum S(h_1)\alpha h_2=\varepsilon(h)\alpha,
			&\qquad
			&\sum h_1 \beta S(h_2)=\varepsilon(h)\beta,\\
			\label{eq:2.6}
			&\sum X^1\beta S(X^2) \alpha X^3=1_H,
			&\qquad
			&\sum S(x^1)\alpha x^2 \beta S(x^3)=1_H.
		\end{alignat}
		We use Sweedler's notation $\Delta(h)=\sum h_1\ot h_2$ and write
		$\Phi=\sum X^1\ot X^2\ot X^3$ and $\Phi^{-1}=\sum x^1\ot x^2\ot x^3$.
		Henceforth, we omit the summation signs in Sweedler's notation and for tensor components of the associator.
	\end{defi}
	
	As usual, we use independent copies of tensor components:
	\begin{equation*}
		\Phi=X^1\ot X^2\ot X^3 = Y^1\ot Y^2\ot Y^3 = \cdots,
		\qquad
		\Phi^{-1}=x^1\ot x^2\ot x^3 =y^1\ot y^2\ot y^3 = \cdots.
	\end{equation*}
	
	Since $\Delta$ is not coassociative, we use the notation
	\begin{equation*}
		(\Delta\ot \Id_H)(\Delta(h))=h_{(1,1)}\ot h_{(1,2)}\ot h_2,
		\qquad
		(\Id_H \ot \Delta)(\Delta(h))=h_1 \ot h_{(2,1)} \ot h_{(2,2)}.
	\end{equation*}
	We write $\Delta^{\cop}(h)=h_2\ot h_1$ for the opposite coproduct.
	
	\begin{defi}
		Let $H=(H,\Delta,\varepsilon,\Phi,\alpha,\beta,S)$ be a quasi-Hopf algebra.
		Define $\gamma,\delta\in H\ot H$ by
		\begin{equation*}
			\gamma:=S(X^2x^1_2)\alpha X^3x^2\ot S(X^1x^1_1)\alpha x^3
			= S(x^1X^2)\alpha x^2X^3_1\ot S(X^1)\alpha x^3X^3_2,
		\end{equation*}
		\begin{equation*}
			\delta:=X^1_1x^1\beta S(X^3)\ot X^1_2x^2\beta S(X^2x^3)
			= x^1\beta S(x^3_2X^3)\ot x^2X^1\beta S(x^3_1X^2).
		\end{equation*}
		The following invertible element $f\in H\ot H$ is called the \emph{Drinfel'd twist}
		associated to $H$:
		\begin{equation*}
			f=(S\ot S)(\Delta^{\cop}(x^1))\gamma \,\Delta(x^2\beta S(x^3)).
		\end{equation*}
		Its inverse is given by
		\begin{equation*}
			f^{-1}=\Delta(S(x^1)\alpha x^2)\,\delta\,(S\ot S)(\Delta^{\cop}(x^3)).
		\end{equation*}
		We also write
		\begin{equation*}
			f = f^{1}\ot f^{2} = F^{1}\ot F^{2}=\cdots,
			\qquad
			f^{-1} = g^{1}\ot g^{2} = G^{1}\ot G^{2}=\cdots.
		\end{equation*}
	\end{defi}
	
	The Drinfel'd twist satisfies \cite{Dr}:
	\begin{equation}\label{eq:2.7}
		f\Delta(S(h))f^{-1}=(S\ot S)(\Delta^{\cop}(h)),\qquad \forall h \in H.
	\end{equation}
	\begin{equation}\label{eq:2.8}
		(1_H\ot f)(\Id_H \ot \Delta)(f)\Phi(\Delta\ot \Id_H)(f^{-1})(f^{-1}\ot 1_H)
		=(S\ot S\ot S)(X^3\ot X^2\ot X^1).
	\end{equation}
	
	The following elements appear frequently in computations:
	\begin{equation}\label{eq:2.9}
		q_R=X^1\ot S^{-1}(\alpha X^3)X^2,\qquad p_L=X^2S^{-1}(X^1\beta)\ot X^3.
	\end{equation}
	We use the notations
	\begin{equation*}
		q_R= q^1\ot q^2 =Q^1\ot Q^2 =\cdots,
		\qquad
		p_L=\tilde{p}^1\ot \tilde{p}^2=\tilde{P}^1\ot \tilde{P}^2=\cdots.
	\end{equation*}
	
	For $q_R,p_L$, the Drinfel'd twist $f$, its inverse $f^{-1}$, and $h\in H$,
	the following identities are known \cites{HN1,HN2}:
	\begin{gather}\label{eq:2.10}
		(1_H\ot S^{-1}(h_2))\,q_R\,\Delta(h_1)=(h\ot 1_H)\,q_R,\\
		\label{eq:2.11}
		\Delta(h_2)\,p_L\,(S^{-1}(h_1)\ot 1_H)=p_L\,(1_H\ot h),\\
		\label{eq:2.12}
		q^{1} Q_{1}^1 x^{1} \ot q^{2} Q_{2}^1 x^{2} \ot Q^{2} x^{3}
		= q^{1} X_{1}^{1} \ot S^{-1}(f^{2} X^{3})q_{1}^{2} X_{(2,1)}^{1}
		\ot S^{-1}(f^{1} X^{2})q_{2}^{2} X_{(2,2)}^{1},\\
		\label{eq:2.13}
		x^{1} \tilde{p}^{1} \ot x^{2} \tilde{p}_{1}^{2} \tilde{P}^{1}
		\ot x^{3} \tilde{p}_{2}^{2} \tilde{P}^{2}
		= X_{(1,1)}^{3} \tilde{p}_1^{1} S^{-1}(X^{2} g^{2})
		\ot X_{(1,2)}^{3} \tilde{p}_{2}^{1} S^{-1}(X^{1} g^{1})
		\ot X_{2}^{3} \tilde{p}^{2}.
	\end{gather}
	
	\section{The first Heisenberg doubles}\label{sec:HDdef}
	
	In this section, we introduce the first Heisenberg doubles $\HD_1(H^\ast)$ and
	$\HD_1(H)$ for a quasi-Hopf algebra $H$. As explained in the Introduction, these two
	algebras must be defined separately. Here, $\HD_1(H^\ast)$ is an algebra in the
	monoidal category $\mathcal{M}_H$ of right $H$-modules, whereas $\HD_1(H)$ is an
	algebra in the monoidal category ${}_H\mathcal{M}$ of left $H$-modules. Unlike in the
	Hopf case, both are generally nonassociative as ordinary algebras.
	
	We use the standard left and right actions of $H$ on $H^\ast$:
	for $h,a\in H$ and $\xi\in H^\ast$,
	\begin{equation*}
		(h\rightharpoonup \xi)(a):=\xi(ah),
		\qquad
		(\xi \leftharpoonup h)(a):=\xi(ha).
	\end{equation*}
	We denote by $\ast$ the convolution product on $H^\ast$.
	
	\begin{defi}
		Let $H$ be a quasi-Hopf algebra.
		Define a multiplication and a right $H$-action on $H^\ast\ot H$ by
		\begin{align}
			\label{eq:HD1Hstar-mult}
			(\xi \# a)(\nu \# b)
			&:=(x^1 \rightharpoonup \xi)\ast(x^2a_1 \rightharpoonup \nu)\#x^3a_2b,
			&&\xi \# a,\nu \# b \in H^\ast\ot H, \\
			\label{eq:HD1Hstar-act}
			(\xi \# a)\triangleleft h
			&:= (\xi \leftharpoonup h)\# a,
			&&\xi \# a\in H^\ast\ot H,\ h \in H.
		\end{align}
		We denote the resulting object by $\HD_1(H^\ast)$ and call it
		the \emph{first Heisenberg double of $H^\ast$}. We write $\xi\ot a$ as $\xi\# a$.
	\end{defi}
	
	In particular,
	\begin{align*}
		(\xi \# a)(\varepsilon \# b)&=\xi\#ab,\\
		(\varepsilon \# a)(\nu \# b)&=(a_1 \rightharpoonup \nu)\#a_2b.
	\end{align*}
	
	\begin{defi}
		Let $H$ be a quasi-Hopf algebra.
		Define a multiplication and a left $H$-action on $H\ot H^\ast$ by
		\begin{align}
			\label{eq:HD1H-mult}
			(a \# \xi)(b \# \nu)
			&:=ab_1x^1\#(\xi \leftharpoonup b_2x^2)\ast(\nu \leftharpoonup x^3),
			&&a \# \xi,b \# \nu \in H\ot H^\ast, \\
			\label{eq:HD1H-act}
			h \triangleright (a \# \xi)
			&:= a \# (h \rightharpoonup \xi),
			&&a \# \xi\in H\ot H^\ast,\ h \in H.
		\end{align}
		We denote the resulting object by $\HD_1(H)$ and call it
		the \emph{first Heisenberg double of $H$}. We write $a\ot \xi$ as $a\#\xi$.
	\end{defi}
	
	In particular,
	\begin{align*}
		(a \# \varepsilon)(b \# \nu)&=ab\#\nu, \\
		(a \# \xi)(b \# \varepsilon)&=ab_1\#(\xi \leftharpoonup b_2).
	\end{align*}
	
	\section{Quasi-pentagon and quasi-Hopf equations}
	
	In this section we prove the corrected pentagon- and Hopf-type identities for the first Heisenberg doubles introduced in Section~\ref{sec:HDdef}. More precisely, in $\HD_1(H^\ast)$ the canonical element $W$ satisfies the quasi-pentagon equation and its quasi-inverse $\tilde{W}$ satisfies the quasi-Hopf equation, whereas in $\HD_1(H)$ the roles are reversed: $\bar{W}$ satisfies the quasi-Hopf equation and $\hat{W}$ satisfies the quasi-pentagon equation. Lemma~\ref{lem:4.1} collects the identities for $U$ and $\tilde{V}$ that drive the computations, while Lemmas~\ref{lem:4.2} and \ref{lem:4.3} ensure that the triple products appearing in the statements are unambiguous.
	
	\begin{lem}\label{lem:4.1}
		Let $H$ be a quasi-Hopf algebra with bijective antipode.
		Define elements $U=U^1\ot U^2$ and $\tilde{V}=\tilde{V}^1\ot \tilde{V}^2$ in $H\ot H$ by
		\begin{equation}\label{eq:4.1}
			U:=g^1S(q^2)\ot g^2S(q^1),
			\qquad
			\tilde{V}:=S(\tilde{p}^2)f^1\ot S(\tilde{p}^1)f^2.
		\end{equation}
		We use the notations
		\begin{equation*}
			U=U^1\ot U^2=u^1\ot u^2=\mathcal{U}^1\ot \mathcal{U}^2=\cdots,
			\qquad
			\tilde{V}=\tilde{V}^1\ot \tilde{V}^2=\tilde{v}^1\ot \tilde{v}^2=\tilde{\mathcal{V}}^1\ot \tilde{\mathcal{V}}^2=\cdots.
		\end{equation*}
		Here $q_R=q^1\ot q^2$ and $p_L=\tilde{p}^1\ot \tilde{p}^2$,
		$f=f^1\ot f^2$ is the Drinfel'd twist, and $f^{-1}=g^1\ot g^2$ is its inverse.
		Then for all $h\in H$,
		\begin{gather}\label{eq:4.2}
			U(1_H \ot S(h))=\Delta(S(h_1))U(h_2 \ot 1_H),\\
			\label{eq:4.3}
			(S(h) \ot 1_H)\tilde{V}=(1_H \ot h_1)\tilde{V}\Delta(S(h_2)),\\
			\label{eq:4.4}
			\Phi^{-1}\,(\Id_H \ot \Delta)(U)\,(1_H \ot U)
			=(\Delta \ot \Id_H)(\Delta(S(X^{1}))U)\,(X^{2} \ot X^{3} \ot 1_H),\\
			\label{eq:4.5}
			(\tilde{V} \ot 1_H)(\Delta \ot \Id_H)(\tilde{V})\Phi^{-1}
			=(1_H \ot X^1 \ot X^2)(\Id_H\ot \Delta)(\tilde{V}\Delta(S(X^3))).
		\end{gather}
	\end{lem}
	
	\begin{proof}
		The identities \eqref{eq:4.2} and \eqref{eq:4.4} are proved in \cite{BCPVO}*{Lemmas~7.32 and 7.45}.
		For the convenience of the reader, we include the computations.
		
		First we prove \eqref{eq:4.2} and \eqref{eq:4.3}.
		\begin{align*}
			\Delta(S(h_1))\,U\,[h_2 \ot 1_H]\;
			&\eqbyref{eq:4.1}
			S(h_1)_1 g^1 S(q^2) h_2 \ot S(h_1)_2 g^2 S(q^1)\\
			&\eqbyref{eq:2.7}
			g^1 S(q^2 h_{(1,2)}) h_2 \ot g^2 S(q^1 h_{(1,1)})\\
			&\eqbyref{eq:2.10}
			g^1 S(q^2) \ot g^2 S(h q^1)
			\eqbyref{eq:4.1}
			U\,[1_H \ot S(h)],
		\end{align*}
		which gives \eqref{eq:4.2}. Similarly,
		\begin{align*}
			[1_H \ot h_1]\tilde{V}\Delta(S(h_2))\;
			&\eqbyref{eq:4.1}
			S(\tilde{p}^2) f^1 S(h_2)_1 \ot h_1S(\tilde{p}^1) f^2 S(h_2)_2\\
			&\eqbyref{eq:2.7}
			S(h_{(2,2)}\tilde{p}^2) f^1 \ot h_1S(h_{(2,1)}\tilde{p}^1) f^2\\
			&\eqbyref{eq:2.11}
			S(\tilde{p}^2h) f^1 \ot S(\tilde{p}^1) f^2
			\eqbyref{eq:4.1}
			[S(h) \ot 1_H]\tilde{V},
		\end{align*}
		which is \eqref{eq:4.3}.
		
		Next we prove \eqref{eq:4.4} and \eqref{eq:4.5}.
		\begin{align*}
			\Phi^{-1}(\Id_H \ot \Delta)(U)(1_H \ot U)\;
			&\eqbyref{eq:4.1}
			x^1 g^1 S(q^2)
			\ot x^2 g_1^2 S(q^1)_1G^1S(Q^2)
			\ot x^3 g_2^2 S(q^1)_2G^2S(Q^1)\\
			&\eqbyref{eq:2.7}
			x^1 g^1 S(q^2)
			\ot x^2 g_1^2 G^1 S(Q^2 q^1_2)
			\ot x^3 g_2^2 G^2 S(Q^1 q^1_1)\\
			&\eqbyref{eq:2.12}
			x^1 g^1 S(q_2^2 X^1_{(2,2)} Y^3)f^1 X^2
			\ot x^2 g_1^2 G^1 S(q^2_1 X^1_{(2,1)} Y^2)f^2 X^3 \\
			&\qquad\ot x^3 g_2^2 G^2 S(q^1 X_1^1 Y^1)\\
			&\eqbyref{eq:2.7}
			x^1 g^1 S(Y^3)f^1S(q^2 X^1_2)_1 X^2
			\ot x^2 g_1^2 G^1 S(Y^2)f^2S(q^2 X^1_2)_2 X^3 \\
			&\qquad\ot x^3 g_2^2 G^2 S(Y^1)S(q^1 X_1^1)\\
			&\eqbyref{eq:2.8}
			g^1_1S(q^2 X^1_2)_1 X^2
			\ot g^1_2 S(q^2 X^1_2)_2 X^3 \ot
			g^2 S(q^1 X_1^1)\\
			&= (\Delta(g^1S(X^1_2)S(q^2))\ot g^2S(X^1_1)S(q^1))(X^2\ot X^3\ot 1_H)\\
			&\eqbyref{eq:2.7}
			(\Delta(S(X^1)_1g^1S(q^2))\ot S(X^1)_2g^2S(q^1))(X^2\ot X^3\ot 1_H)\\
			&\eqbyref{eq:4.1}
			(\Delta \ot \Id_H)(\Delta(S(X^1))U)(X^2 \ot X^3 \ot 1_H),
		\end{align*}
		which gives \eqref{eq:4.4}. Finally,
		\begin{align*}
			(\tilde{V} \ot 1_H)(\Delta \ot \Id_H)(\tilde{V})\Phi^{-1}\;
			&\eqbyref{eq:4.1}
			S(\tilde{p}^2)f^1S(\tilde{P}^2)_1F^1_1x^1
			\ot S(\tilde{p}^1)f^2S(\tilde{P}^2)_2F^1_2x^2
			\ot S(\tilde{P}^1)F^2x^3\\
			&\eqbyref{eq:2.7}
			S(\tilde{P}^2_2\tilde{p}^2)f^1F^1_1x^1
			\ot S(\tilde{P}^2_1\tilde{p}^1)f^2F^1_2x^2
			\ot S(\tilde{P}^1)F^2x^3\\
			&\eqbyref{eq:2.13}
			S(Y^3X^3_2\tilde{p}^2)f^1F^1_1x^1
			\ot X^1g^1S(Y^2X^3_{(1,2)}\tilde{p}^1_2)f^2F^1_2x^2\\
			&\qquad\ot X^2g^2S(Y^1X^3_{(1,1)}\tilde{p}^1_1)F^2x^3\\
			&\eqbyref{eq:2.7}
			S(X^3_2\tilde{p}^2)S(Y^3)f^1F^1_1x^1
			\ot X^1S(X^3_1\tilde{p}^1)_1g^1S(Y^2)f^2F^1_2x^2\\
			&\qquad\ot X^2S(X^3_1\tilde{p}^1)_2g^2S(Y^1)F^2x^3\\
			&\eqbyref{eq:2.8}
			S(X^3_2\tilde{p}^2)f^1
			\ot X^1S(X^3_1\tilde{p}^1)_1f^2_1
			\ot X^2S(X^3_1\tilde{p}^1)_2f^2_2\\
			&=(1_H\ot X^1\ot X^2)\Bigl(S(X^3_2\tilde{p}^2)f^1\ot \Delta(S(X^3_1\tilde{p}^1)f^2)\Bigr)\\
			&\eqbyref{eq:2.7}
			(1_H\ot X^1\ot X^2)\Bigl(S(\tilde{p}^2)f^1S(X^3)_1\ot \Delta(S(\tilde{p}^1)f^2S(X^3)_2)\Bigr)\\
			&\eqbyref{eq:4.1}
			(1_H \ot X^1 \ot X^2)(\Id_H\ot \Delta)(\tilde{V}\Delta(S(X^3))),
		\end{align*}
		which is \eqref{eq:4.5}.
	\end{proof}
	
	\begin{lem}\label{lem:4.2}
		Let $H$ be a quasi-Hopf algebra and let $\xi \# a,\nu \# b,\zeta \# c\in \HD_1(H^\ast)$.
		If one of $\xi,\nu,\zeta$ equals the counit $\varepsilon$, then
		\begin{equation*}
			\left( (\xi\#a)(\nu\#b)\right) (\zeta\#c)=(\xi\#a)\left( (\nu\#b)(\zeta\#c)\right) .
		\end{equation*}
	\end{lem}
	\begin{proof}
		This follows immediately from the definition of the product in $\HD_1(H^\ast)$
		together with \eqref{eq:2.1}.
	\end{proof}
	
	\begin{lem}\label{lem:4.3}
		Let $H$ be a quasi-Hopf algebra and let $a \# \xi,b \# \nu,c \# \zeta\in \HD_1(H)$.
		If one of $\xi,\nu,\zeta$ equals the counit $\varepsilon$, then
		\begin{equation*}
			\left( (a\#\xi)(b\#\nu)\right) (c\#\zeta)=(a\#\xi)\left( (b\#\nu)(c\#\zeta)\right) .
		\end{equation*}
	\end{lem}
	\begin{proof}
		This follows immediately from the definition of the product in $\HD_1(H)$
		together with \eqref{eq:2.1}.
	\end{proof}
	
	\begin{thm}\label{thm:4.4}
		Let $H$ be a finite-dimensional quasi-Hopf algebra over $k$.
		Let $\{\bm{e}_i\}_{i\in I}$ be a basis of $H$ and $\{\bm{e}^i\}_{i\in I}$ its dual basis.
		Let
		\begin{equation*}
			W=\sum_{i \in I} \varepsilon \# \bm{e}_i \ot \bm{e}^i \# 1_H,
			\qquad
			\tilde{W}=\sum_{i \in I} \varepsilon \# S(\bm{e}_i U^1) \ot \bm{e}^i \# U^2
			\in {\HD_1(H^\ast)}^{\ot 2},
		\end{equation*}
		and
		\begin{equation*}
			\boldsymbol{\Phi}^{-1}=\varepsilon \# x^1 \ot \varepsilon \# x^2 \ot \varepsilon \# x^3,
			\qquad
			\boldsymbol{\Phi}^{321}_{S}=\varepsilon \# S(X^3) \ot \varepsilon \# S(X^2) \ot \varepsilon \# S(X^1)
			\in {\HD_1(H^\ast)}^{\ot 3}.
		\end{equation*}
		Then
		\begin{align}\label{eq:4.6}
			W_{12}W_{13}W_{23}&=W_{23}W_{12}\boldsymbol{\Phi}^{-1},\\
			\label{eq:4.7}
			\tilde{W}_{23}\tilde{W}_{13}\tilde{W}_{12}&=\boldsymbol{\Phi}^{321}_{S}\tilde{W}_{12}\tilde{W}_{23}.
		\end{align}
	\end{thm}
	
	\begin{proof}
		Although the product on $\HD_1(H^\ast)$ is not associative, Lemma~\ref{lem:4.2}
		implies that $(W_{12}W_{13})W_{23}=W_{12}(W_{13}W_{23})$.
		Hence the expression $W_{12}W_{13}W_{23}$ is unambiguous, and the same convention applies to the other products.
		
		We first prove \eqref{eq:4.6}.
		\begin{align*}
			W_{12}W_{13}W_{23}
			&=
			\sum_{i,j,k \in I} \left( (\varepsilon \# \bm{e}_i \ot \bm{e}^i \# 1_H \ot \varepsilon \# 1_H)
			(\varepsilon \# \bm{e}_j \ot \varepsilon \# 1_H  \ot\bm{e}^j \# 1_H )\right)
			(\varepsilon \# 1_H\ot \varepsilon \# \bm{e}_k \ot \bm{e}^k \# 1_H )\\
			&=
			\sum_{i,j,k \in I} (\varepsilon \# \bm{e}_i\bm{e}_j \ot \bm{e}^i \# 1_H \ot \bm{e}^j \# 1_H)
			(\varepsilon \# 1_H\ot \varepsilon \# \bm{e}_k \ot \bm{e}^k \# 1_H )\\
			&=
			\sum_{i,j,k \in I} \varepsilon \# \bm{e}_i\bm{e}_j \ot \bm{e}^i \# \bm{e}_k
			\ot (x^1 \rightharpoonup \bm{e}^j )\ast(x^2 \rightharpoonup \bm{e}^k) \# x^3\\
			&=
			\sum_{i,j,k,l \in I} \varepsilon \# \bm{e}_i\bm{e}_j \ot \bm{e}^i \# \bm{e}_k
			\ot \bm{e}^j((\bm{e}_l)_1x^1)\bm{e}^k((\bm{e}_l)_2x^2)\bm{e}^l \# x^3\\
			&=
			\sum_{i,l \in I} \varepsilon \# \bm{e}_i(\bm{e}_l)_1x^1 \ot \bm{e}^i \# (\bm{e}_l)_2x^2 \ot \bm{e}^l \# x^3.
		\end{align*}
		Similarly,
		\begin{align*}
			W_{23}W_{12}\boldsymbol{\Phi}^{-1}
			&=
			\sum_{i,j \in I} \left( (\varepsilon \# 1_H\ot \varepsilon \# \bm{e}_i \ot \bm{e}^i \# 1_H )
			(\varepsilon \# \bm{e}_j \ot \bm{e}^j \# 1_H \ot \varepsilon \# 1_H)\right)
			(\varepsilon \# x^1 \ot \varepsilon \# x^2 \ot \varepsilon \# x^3)\\
			&=
			\sum_{i,j \in I} (\varepsilon \# \bm{e}_j\ot ((\bm{e}_i)_1\rightharpoonup \bm{e}^j) \# (\bm{e}_i)_2
			\ot \bm{e}^i \# 1_H )
			(\varepsilon \# x^1 \ot \varepsilon \# x^2 \ot \varepsilon \# x^3)\\
			&=
			\sum_{i,j \in I} \varepsilon \# \bm{e}_jx^1\ot ((\bm{e}_i)_1\rightharpoonup \bm{e}^j) \# (\bm{e}_i)_2x^2 \ot \bm{e}^i \# x^3\\
			&=
			\sum_{i,j \in I} \varepsilon \# \bm{e}_j(\bm{e}_i)_1x^1\ot \bm{e}^j \# (\bm{e}_i)_2x^2 \ot \bm{e}^i \# x^3.
		\end{align*}
		Thus \eqref{eq:4.6} holds.
		
		We next prove \eqref{eq:4.7}.
		\begin{align*}
			\boldsymbol{\Phi}^{321}_{S}\tilde{W}_{12}\tilde{W}_{23}
			&=
			\sum_{i,j \in I} (\varepsilon \# S(X^3) \ot \varepsilon \# S(X^2) \ot \varepsilon \# S(X^1)) \\
			&\qquad\times\Bigl(
			( \varepsilon \# S(\bm{e}_i U^1) \ot \bm{e}^i \# U^2 \ot \varepsilon \# 1_H)
			(\varepsilon \# 1_H \ot \varepsilon \# S(\bm{e}_j u^1) \ot \bm{e}^j \# u^2)
			\Bigr)\\
			&=
			\sum_{i,j \in I} (\varepsilon \# S(X^3) \ot \varepsilon \# S(X^2) \ot \varepsilon \# S(X^1))
			( \varepsilon \# S(\bm{e}_i U^1) \ot \bm{e}^i \# U^2S(\bm{e}_j u^1) \ot \bm{e}^j \# u^2)\\
			&=
			\sum_{i,j \in I} \varepsilon \# S(\bm{e}_i U^1X^3)
			\ot (S(X^2)_1\rightharpoonup \bm{e}^i) \# S(X^2)_2U^2S(\bm{e}_j u^1) \ot (S(X^1)_1\rightharpoonup \bm{e}^j) \# S(X^1)_2u^2\\
			&=
			\sum_{i,j \in I} \varepsilon \# S(\bm{e}_i S(X^2)_1 U^1X^3)
			\ot \bm{e}^i \# S(X^2)_2U^2S(\bm{e}_j S(X^1)_1 u^1) \ot \bm{e}^j \# S(X^1)_2u^2.
		\end{align*}
		
		On the other hand,
		\begin{align*}
			\tilde{W}_{23}\tilde{W}_{13}\tilde{W}_{12}
			&=
			\sum_{i,j,k \in I} \Bigl(
			(\varepsilon \# 1_H \ot \varepsilon \# S(\bm{e}_i U^1) \ot \bm{e}^i \# U^2)
			(\varepsilon \# S(\bm{e}_j u^1) \ot \varepsilon \# 1_H  \ot\bm{e}^j \# u^2 )
			\Bigr) \\
			&\qquad\times ( \varepsilon \# S(\bm{e}_k \mathcal{U}^1)\ot \bm{e}^k \# \mathcal{U}^2 \ot \varepsilon \# 1_H)\\
			&=
			\sum_{i,j,k \in I} \Bigl(
			\varepsilon \# S(\bm{e}_j u^1) \ot \varepsilon \# S(\bm{e}_i U^1) \ot (x^1 \rightharpoonup \bm{e}^i )\ast(x^2U^2_1 \rightharpoonup \bm{e}^j) \# x^3U^2_2u^2
			\Bigr) \\
			&\qquad\times ( \varepsilon \# S(\bm{e}_k \mathcal{U}^1)\ot \bm{e}^k \# \mathcal{U}^2 \ot \varepsilon \# 1_H)\\
			&=
			\sum_{i,j,k \in I} \varepsilon \# S(\bm{e}_k \mathcal{U}^1\bm{e}_j u^1)
			\ot (S(\bm{e}_i U^1)_1\rightharpoonup \bm{e}^k) \# S(\bm{e}_i U^1)_2\mathcal{U}^2 \\
			&\qquad\ot (x^1 \rightharpoonup \bm{e}^i )\ast(x^2U^2_1 \rightharpoonup \bm{e}^j) \# x^3U^2_2u^2\\
			&=
			\sum_{i,j,k,l \in I} \varepsilon \# S(\bm{e}_k S(\bm{e}_i U^1)_1 \mathcal{U}^1\bm{e}_j u^1)
			\ot \bm{e}^k \# S(\bm{e}_i U^1)_2\mathcal{U}^2 \\
			&\qquad\ot \bm{e}^i((\bm{e}_l)_1x^1) \bm{e}^j((\bm{e}_l)_2x^2U^2_1)\bm{e}^l \# x^3U^2_2u^2 \\
			&=
			\sum_{k,l \in I} \varepsilon \# S(\bm{e}_k S((\bm{e}_l)_1x^1 U^1)_1 \mathcal{U}^1(\bm{e}_l)_2x^2U^2_1 u^1)
			\ot \bm{e}^k \# S((\bm{e}_l)_1x^1 U^1)_2\mathcal{U}^2 \ot \bm{e}^l \# x^3U^2_2u^2 \\
			&\eqbyref{eq:4.4}
			\sum_{k,l \in I} \varepsilon \# S(\bm{e}_k S((\bm{e}_l)_1S(X^1)_{(1,1)}U^1_1X^2)_1 \mathcal{U}^1(\bm{e}_l)_2S(X^1)_{(1,2)}U^1_2X^3) \\
			&\qquad\ot \bm{e}^k \# S((\bm{e}_l)_1S(X^1)_{(1,1)}U^1_1X^2)_2\mathcal{U}^2 \ot \bm{e}^l \# S(X^1)_2U^2 \\
			&=
			\sum_{k,l \in I} \varepsilon \# S(\bm{e}_k S(X^2)_1S((\bm{e}_lS(X^1)_1U^1)_1)_1 \mathcal{U}^1(\bm{e}_lS(X^1)_1U^1)_2X^3) \\
			&\qquad\ot \bm{e}^k \# S(X^2)_2S((\bm{e}_lS(X^1)_1U^1)_1)_2\mathcal{U}^2 \ot \bm{e}^l \# S(X^1)_2U^2 \\
			&\eqbyref{eq:4.2}
			\sum_{k,l \in I} \varepsilon \# S(\bm{e}_k S(X^2)_1 \mathcal{U}^1X^3)
			\ot \bm{e}^k \# S(X^2)_2\mathcal{U}^2S(\bm{e}_lS(X^1)_1U^1) \ot \bm{e}^l \# S(X^1)_2U^2 .
		\end{align*}
		
		Thus \eqref{eq:4.7} holds.
	\end{proof}
	
	\begin{thm}\label{thm:4.5}
		Under the same assumptions as in Theorem~\ref{thm:4.4}, let
		\begin{equation*}
			\bar{W}=\sum_{i \in I} \bm{e}_i \# \varepsilon \ot 1_H \# \bm{e}^i,
			\qquad
			\hat{W}=\sum_{i \in I} S(\tilde{V}^2 \bm{e}_i)\# \varepsilon \ot \tilde{V}^1 \# \bm{e}^i
			\in {\HD_1(H)}^{\ot 2},
		\end{equation*}
		and
		\begin{equation*}
			(\overline{\boldsymbol{\Phi}}^{-1})^{321}=x^3 \# \varepsilon \ot x^2 \# \varepsilon \ot x^1 \# \varepsilon,
			\qquad
			\overline{\boldsymbol{\Phi}}_S=S(X^1) \# \varepsilon \ot S(X^2) \# \varepsilon \ot S(X^3) \# \varepsilon
			\in {\HD_1(H)}^{\ot 3}.
		\end{equation*}
		Then
		\begin{align}\label{eq:4.8}
			\bar{W}_{23}\bar{W}_{13}\bar{W}_{12}&=(\overline{\boldsymbol{\Phi}}^{-1})^{321}\bar{W}_{12}\bar{W}_{23},\\
			\label{eq:4.9}
			\hat{W}_{12}\hat{W}_{13}\hat{W}_{23}&=\hat{W}_{23}\hat{W}_{12}\overline{\boldsymbol{\Phi}}_S.
		\end{align}
	\end{thm}
	
	\begin{proof}
		As in Theorem~\ref{thm:4.4}, Lemma~\ref{lem:4.3} implies that the parenthesization
		in the triple products is unambiguous in the present situation.
		
		We first prove \eqref{eq:4.8}.
		\begin{align*}
			\bar{W}_{23}\bar{W}_{13}\bar{W}_{12}
			&=
			\sum_{i,j,k \in I} \Bigl(
			(1_H \# \varepsilon \ot\bm{e}_i \# \varepsilon \ot 1_H \# \bm{e}^i)
			(\bm{e}_j \# \varepsilon \ot 1_H \# \varepsilon \ot 1_H \# \bm{e}^j)
			\Bigr)
			(\bm{e}_k \# \varepsilon \ot 1_H \# \bm{e}^k \ot 1_H \# \varepsilon )\\
			&=
			\sum_{i,j,k \in I} \Bigl(
			\bm{e}_j \# \varepsilon \ot\bm{e}_i \# \varepsilon \ot x^1 \# (\bm{e}^i\leftharpoonup x^2)\ast(\bm{e}^j \leftharpoonup x^3)
			\Bigr)
			(\bm{e}_k \# \varepsilon \ot 1_H \# \bm{e}^k \ot 1_H \# \varepsilon )\\
			&=
			\sum_{i,j,k,l \in I} \bm{e}_j\bm{e}_k \# \varepsilon \ot\bm{e}_i \# \bm{e}^k
			\ot x^1 \# \bm{e}^i(x^2(\bm{e}_l)_1)\bm{e}^j(x^3(\bm{e}_l)_2)\bm{e}^l\\
			&=
			\sum_{k,l \in I} x^3(\bm{e}_l)_2\bm{e}_k \# \varepsilon \ot x^2(\bm{e}_l)_1 \# \bm{e}^k \ot x^1 \# \bm{e}^l.
		\end{align*}
		On the other hand,
		\begin{align*}
			(\overline{\boldsymbol{\Phi}}^{-1})^{321}\bar{W}_{12}\bar{W}_{23}
			&=
			\sum_{i,j \in I} (x^3 \# \varepsilon \ot x^2 \# \varepsilon \ot x^1 \# \varepsilon)
			\Bigl(
			(\bm{e}_i \# \varepsilon \ot 1_H \# \bm{e}^i \ot 1_H \# \varepsilon )
			(1_H \# \varepsilon \ot\bm{e}_j \# \varepsilon \ot 1_H \# \bm{e}^j)
			\Bigr)\\
			&=
			\sum_{i,j \in I} (x^3 \# \varepsilon \ot x^2 \# \varepsilon \ot x^1 \# \varepsilon)
			(\bm{e}_i \# \varepsilon \ot (\bm{e}_j)_1 \# (\bm{e}^i \leftharpoonup (\bm{e}_j)_2) \ot 1_H \# \bm{e}^j)\\
			&=
			\sum_{i,j \in I} x^3\bm{e}_i \# \varepsilon \ot x^2(\bm{e}_j)_1 \# (\bm{e}^i \leftharpoonup (\bm{e}_j)_2) \ot x^1 \# \bm{e}^j\\
			&=
			\sum_{i,j \in I} x^3(\bm{e}_j)_2\bm{e}_i \# \varepsilon \ot x^2(\bm{e}_j)_1 \# \bm{e}^i \ot x^1 \# \bm{e}^j.
		\end{align*}
		Thus \eqref{eq:4.8} holds.
		
		We next prove \eqref{eq:4.9}.
		\begin{align*}
			\hat{W}_{12}\hat{W}_{13}\hat{W}_{23}
			&=
			\sum_{i,j,k \in I} \Bigl(
			(S(\tilde{V}^2 \bm{e}_i)\# \varepsilon \ot \tilde{V}^1 \# \bm{e}^i \ot 1_H \# \varepsilon)
			(S(\tilde{v}^2 \bm{e}_j)\# \varepsilon \ot 1_H \# \varepsilon \ot \tilde{v}^1 \# \bm{e}^j)
			\Bigr)\\
			&\qquad\times(1_H \# \varepsilon \ot S(\tilde{\mathcal{V}}^2 \bm{e}_k)\# \varepsilon \ot  \tilde{\mathcal{V}}^1 \# \bm{e}^k)\\
			&=
			\sum_{i,j,k \in I}
			\Bigl(
			S(\tilde{v}^2 \bm{e}_j\tilde{V}^2 \bm{e}_i)\# \varepsilon \ot \tilde{V}^1 \# \bm{e}^i \ot \tilde{v}^1 \# \bm{e}^j
			\Bigr)
			(\varepsilon \# 1_H \ot S(\tilde{\mathcal{V}}^2 \bm{e}_k)\# \varepsilon \ot  \tilde{\mathcal{V}}^1 \# \bm{e}^k)\\
			&=
			\sum_{i,j,k \in I}
			S(\tilde{v}^2 \bm{e}_j\tilde{V}^2 \bm{e}_i)\# \varepsilon
			\ot \tilde{V}^1S(\tilde{\mathcal{V}}^2 \bm{e}_k)_1 \# (\bm{e}^i \leftharpoonup S(\tilde{\mathcal{V}}^2 \bm{e}_k)_2)\\
			&\qquad\ot \tilde{v}^1\tilde{\mathcal{V}}^1_1x^1 \# (\bm{e}^j \leftharpoonup \tilde{\mathcal{V}}^1_2x^2)\ast(\bm{e}^k \leftharpoonup x^3)\\
			&=
			\sum_{i,j,k,l \in I}
			S(\tilde{v}^2 \bm{e}_j\tilde{V}^2 S(\tilde{\mathcal{V}}^2 \bm{e}_k)_2\bm{e}_i)\# \varepsilon
			\ot \tilde{V}^1S(\tilde{\mathcal{V}}^2 \bm{e}_k)_1 \# \bm{e}^i\\
			&\qquad\ot \tilde{v}^1\tilde{\mathcal{V}}^1_1x^1 \# \bm{e}^j(\tilde{\mathcal{V}}^1_2x^2(\bm{e}_l)_1)\bm{e}^k(x^3(\bm{e}_l)_2)\bm{e}^l\\
			&=
			\sum_{i,l \in I}
			S(\tilde{v}^2 \tilde{\mathcal{V}}^1_2x^2(\bm{e}_l)_1\tilde{V}^2 S(\tilde{\mathcal{V}}^2 x^3(\bm{e}_l)_2)_2\bm{e}_i)\# \varepsilon
			\ot \tilde{V}^1S(\tilde{\mathcal{V}}^2 x^3(\bm{e}_l)_2)_1 \# \bm{e}^i \ot \tilde{v}^1\tilde{\mathcal{V}}^1_1x^1 \# \bm{e}^l\\
			&\eqbyref{eq:4.5}
			\sum_{i,l \in I}
			S(X^1\tilde{v}^2_1S(X^3)_{(2,1)}(\bm{e}_l)_1\tilde{V}^2 S(X^2\tilde{v}^2_2S(X^3)_{(2,2)}(\bm{e}_l)_2)_2\bm{e}_i)\# \varepsilon\\
			&\qquad\ot \tilde{V}^1S(X^2\tilde{v}^2_2S(X^3)_{(2,2)}(\bm{e}_l)_2)_1 \# \bm{e}^i\ot \tilde{v}^1S(X^3)_1 \# \bm{e}^l\\
			&=
			\sum_{i,l \in I}
			S(X^1(\tilde{v}^2S(X^3)_2\bm{e}_l)_1\tilde{V}^2 S((\tilde{v}^2S(X^3)_2\bm{e}_l)_2)_2S(X^2)_2\bm{e}_i)\# \varepsilon\\
			&\qquad\ot \tilde{V}^1S((\tilde{v}^2S(X^3)_2\bm{e}_l)_2)_1S(X^2)_1 \# \bm{e}^i\ot \tilde{v}^1S(X^3)_1 \# \bm{e}^l\\
			&\eqbyref{eq:4.3}
			\sum_{i,l \in I}
			S(X^1\tilde{V}^2 S(X^2)_2\bm{e}_i)\# \varepsilon \ot S(\tilde{v}^2S(X^3)_2\bm{e}_l)\tilde{V}^1S(X^2)_1 \# \bm{e}^i\ot \tilde{v}^1S(X^3)_1 \# \bm{e}^l.
		\end{align*}
		
		On the other hand,
		\begin{align*}
			\hat{W}_{23}\hat{W}_{12}\overline{\boldsymbol{\Phi}}_S
			&=
			\sum_{i,j \in I} \Bigl(
			(1_H \# \varepsilon \ot S(\tilde{V}^2 \bm{e}_i)\# \varepsilon \ot  \tilde{V}^1 \# \bm{e}^i)
			(S(\tilde{v}^2 \bm{e}_j)\# \varepsilon \ot \tilde{v}^1 \# \bm{e}^j \ot 1_H \# \varepsilon)
			\Bigr)\\
			&\qquad\times(S(X^1) \# \varepsilon \ot S(X^2) \# \varepsilon \ot S(X^3) \# \varepsilon)\\
			&=
			\sum_{i,j \in I}
			\Bigl(
			S(\tilde{v}^2 \bm{e}_j)\# \varepsilon \ot S(\tilde{V}^2 \bm{e}_i)\tilde{v}^1\# \bm{e}^j \ot  \tilde{V}^1 \# \bm{e}^i
			\Bigr)
			(S(X^1) \# \varepsilon \ot S(X^2) \# \varepsilon \ot S(X^3) \# \varepsilon)\\
			&=
			\sum_{i,j \in I}
			S(X^1\tilde{v}^2 \bm{e}_j)\# \varepsilon \ot S(\tilde{V}^2 \bm{e}_i)\tilde{v}^1S(X^2)_1\# (\bm{e}^j \leftharpoonup S(X^2)_2)
			\ot  \tilde{V}^1S(X^3)_1 \# (\bm{e}^i\leftharpoonup S(X^3)_2)\\
			&=
			\sum_{i,j \in I}
			S(X^1\tilde{v}^2 S(X^2)_2\bm{e}_j)\# \varepsilon \ot S(\tilde{V}^2 S(X^3)_2\bm{e}_i)\tilde{v}^1S(X^2)_1\# \bm{e}^j
			\ot  \tilde{V}^1S(X^3)_1 \# \bm{e}^i.
		\end{align*}
		Thus \eqref{eq:4.9} holds.
	\end{proof}
	
	\section{Explicit examples for twisted Heisenberg doubles}\label{sec:twisted}
	
	In this section we present explicit examples of the quasi-pentagon equation, the quasi-Hopf equation, and the quasi-inverses for twisted Heisenberg doubles associated with finite groups and normalized $3$-cocycles. We also explain that, unlike in the Hopf case, the canonical element may fail to admit an inverse. We refer to Dijkgraaf, Pasquier, and Roche \cite{DPR} for the associated twisted quantum double $D^{\omega}(G)$. Throughout this section, let $G$ be a finite group and let $\omega$ be a normalized $3$-cocycle on $G$. Then, for all $a,b,c,d\in G$:
	\begin{equation}\label{eq:5.1}
		\omega(a,b,c)\omega(a,bc,d)\omega(b,c,d)=\omega(ab,c,d)\omega(a,b,cd),
	\end{equation}
	and $\omega(a,b,c)=1$ whenever at least one of $a,b,c$ is the identity element.

	Let $k[G]$ be the group algebra of $G$ over $k$, and let $k(G)$ be the algebra of $k$-valued functions on $G$. For $a\in G$, let $\delta_a\in k(G)$ be the delta function defined by $\delta_a(b)=\delta_{a,b}$. Note that the identity element of the function algebra $k(G)$ is given by $1_{k(G)} = \sum_{a \in G} \delta_a$. We write $k^{\omega}(G)$ for the vector space $k(G)$ endowed with the following quasi-Hopf algebra structure:
	\begin{equation*}
	\begin{aligned}
		\delta_a \delta_b &= \delta_{a,b}\delta_b,\qquad \Delta(\delta_a)=\sum_{g,h \in G, \; a=gh}\delta_g\otimes\delta_h,\qquad \varepsilon(\delta_a)=\delta_a(1_G), \\
		\Phi &= \sum_{a,b,c \in G} \dfrac{1}{\omega(a,b,c)}\delta_a \otimes \delta_b \otimes \delta_c,\qquad \Phi^{-1}=\sum_{a,b,c \in G} \omega(a,b,c)\delta_a \otimes \delta_b \otimes \delta_c, \\
		S(\delta_a)&=\delta_{a^{-1}},\qquad \alpha = 1_{k(G)},\qquad \beta=\sum_{a \in G}\omega(a,a^{-1},a)\delta_a.
	\end{aligned}
\end{equation*}
	We write
	\[
	\HD_1^{\omega}(G^{\ast}) := \HD_1\bigl({k^{\omega}(G)}^{\ast}\bigr),
	\qquad
	\HD_1^{\omega}(G) := \HD_1\bigl(k^{\omega}(G)\bigr),
	\]
	and refer to them as twisted Heisenberg doubles. As vector spaces, there are canonical identifications
	\[
	\HD_1^{\omega}(G^{\ast})\cong k[G]\otimes k(G),
	\qquad
	\HD_1^{\omega}(G)\cong k(G)\otimes k[G].
	\]
	Under the first identification, we write elements as $g\#\delta_a$, while under the second identification, we write them as $\delta_a\# g$.
	
	The multiplication on $\HD_1^{\omega}(G^{\ast})$ is given by
	\[
	(g\# \delta_a)(h\# \delta_b)=\delta_{a,hb}\omega(g,h,b)gh\#\delta_b,
	\]
	and the right $k^{\omega}(G)$-action is given by
	\[
	(g\#\delta_a)\triangleleft \delta_b=\delta_{b,g}g\#\delta_a.
	\]
	Similarly, the multiplication on $\HD_1^{\omega}(G)$ is given by
	\[
	(\delta_a \# g)(\delta_b \# h)=\delta_{b,ag}\omega(a,g,h)\delta_a\# gh,
	\]
	and the left $k^{\omega}(G)$-action is given by
	\[
	\delta_b \triangleright (\delta_a \# g)=\delta_{b,g}\delta_a\# g.
	\]
	
	For $k^{\omega}(G)$, the elements $U=U^1\ot U^2$ and $\tilde{V}=\tilde{V}^1\ot \tilde{V}^2$ defined in Lemma~\ref{lem:4.1} are given by
	\begin{equation*}
	\begin{aligned}
		U &= U^1\otimes U^2 = \sum_{a,b \in G}\frac{1}{\omega(b^{-1}a^{-1},a,b)} \,\delta_a \otimes \delta_b,\\
		\tilde{V} &= \tilde{V}^1\otimes \tilde{V}^2 = \sum_{a,b \in G}\frac{\omega(b^{-1}a^{-1},a,b)}{\omega(b^{-1},a^{-1},a)\omega(b^{-1},b,b^{-1}a^{-1})}\,\delta_a \otimes \delta_b.
	\end{aligned}
\end{equation*}
	Under the above identifications, the counit of $k^\omega(G)$ corresponds to $1_G \in k[G]$, while the unit of $k^\omega(G)$ is $1_{k(G)}$. Hence the canonical element $W\in {\HD_1^{\omega}(G^{\ast})}^{\ot 2}$ and its quasi-inverse $\tilde{W}$ are given by
	\begin{equation*}
	\begin{aligned}
		W&=\sum_{g \in G}(1_G\# \delta_g) \otimes (g\# 1_{k(G)}),\\
		\tilde{W}&=\sum_{a,b \in G}\frac{1}{\omega(b^{-1}a,a^{-1},b)}(1_G\# \delta_a) \otimes (a^{-1}\# \delta_b).
	\end{aligned}
\end{equation*}
	Likewise, the canonical element $\bar{W}\in {\HD_1^{\omega}(G)}^{\ot 2}$ and its quasi-inverse $\hat{W}$ are
	\begin{equation*}
	\begin{aligned}
		\bar{W}&=\sum_{g \in G}(\delta_g\# 1_G) \otimes (1_{k(G)}\# g),\\
		\hat{W}&=\sum_{a,b \in G}\frac{\omega(ab^{-1},b,a^{-1})}{\omega(a,b^{-1},b)\omega(a,a^{-1},ab^{-1})}(\delta_a\# 1_G) \otimes (\delta_b \# a^{-1}).
	\end{aligned}
\end{equation*}
	Finally,
	\begin{equation*}
	\begin{aligned}
		\boldsymbol{\Phi}^{-1}
		&=\sum_{a,b,c \in G}\omega(a,b,c)1_G \# \delta_a \otimes 1_G \# \delta_b \otimes 1_G \# \delta_c,\\
		\boldsymbol{\Phi}^{321}_{S}
		&=\sum_{a,b,c \in G}\frac{1}{\omega(c^{-1},b^{-1},a^{-1})}1_G \# \delta_a \otimes 1_G \# \delta_b \otimes 1_G \# \delta_c,\\
		(\overline{\boldsymbol{\Phi}}^{-1})^{321}
		&=\sum_{a,b,c \in G}\omega(c,b,a)\delta_a\# 1_G \otimes \delta_b \# 1_G \otimes \delta_c \# 1_G,\\
		\overline{\boldsymbol{\Phi}}_S
		&=\sum_{a,b,c \in G}\frac{1}{\omega(a^{-1},b^{-1},c^{-1})}\delta_a\# 1_G \otimes \delta_b \# 1_G \otimes \delta_c \# 1_G.
	\end{aligned}
\end{equation*}
	In particular, $\tilde{W}$ need not be the inverse of $W$, and $\hat{W}$ need not be the inverse of $\bar{W}$.

	Theorems~\ref{thm:4.4} and \ref{thm:4.5} specialize to explicit quasi-pentagon and quasi-Hopf equations for twisted Heisenberg doubles. For example, \eqref{eq:4.9} takes the form
	{\small
	\begin{align*}
		\hat{W}_{12}\hat{W}_{13}\hat{W}_{23}
		&=\sum_{a,b,c \in G}\frac{\omega(ab^{-1},b,a^{-1})\omega(ac^{-1},c,a^{-1})\omega(bc^{-1},ca^{-1},ab^{-1})\omega(c,a^{-1},ab^{-1})}{\omega(a,b^{-1},b)\omega(a,a^{-1},ab^{-1})\omega(a,c^{-1},c)\omega(a,a^{-1},ac^{-1})\omega(ba^{-1},ac^{-1},ca^{-1})\omega(ba^{-1},ab^{-1},bc^{-1})}\\
		&\qquad \times \delta_a \# 1_G \otimes \delta_b \# a^{-1} \otimes \delta_c \# b^{-1},\\
		\hat{W}_{23}\hat{W}_{12}\overline{\boldsymbol{\Phi}}_S
		&=\sum_{a,b,c \in G}\frac{\omega(bc^{-1},c,b^{-1})\omega(ab^{-1},b,a^{-1})}{\omega(b,c^{-1},c)\omega(b,b^{-1},bc^{-1})\omega(a,b^{-1},b)\omega(a,a^{-1},ab^{-1})\omega(a^{-1},ab^{-1},bc^{-1})}\\
		&\qquad \times \delta_a \# 1_G \otimes \delta_b \# a^{-1} \otimes \delta_c \# b^{-1}.
	\end{align*}
	}
	These two expressions agree by repeated use of the $3$-cocycle identity \eqref{eq:5.1} and normalization.

\begin{rem}
		Suppose that $W$ is invertible, with inverse $V$. Then a straightforward argument shows that there exists a function $\varphi\co G\times G\to k$ such that
		\[
		V=\sum_{a,b\in G} \varphi(a,b)1_G\# \delta_a \otimes a^{-1}\# \delta_b.
		\]
		Computing $WV$ and $VW$, we obtain
		\begin{align*}
			WV&=\sum_{g,a,b\in G}\varphi(a,b)1_G \# \delta_g \delta_a\otimes \omega(g,a^{-1},b)ga^{-1}\#\delta_b\\
			&=\sum_{a,b\in G}\varphi(a,b)\omega(a,a^{-1},b)1_G\#\delta_a \otimes 1_G \# \delta_b,\\
			VW&=\sum_{g,a,b\in G}\varphi(a,b)1_G \# \delta_a\delta_g\otimes \omega(a^{-1},g,g^{-1}b)a^{-1}g\#\delta_{g^{-1}b}\\
			&=\sum_{a,b\in G}\varphi(a,b)\omega(a^{-1},a,a^{-1}b)1_G\#\delta_a \otimes 1_G \# \delta_{a^{-1}b}\\
			&=\sum_{a,b'\in G}\varphi(a,ab')\omega(a^{-1},a,b')1_G\#\delta_a \otimes 1_G \# \delta_{b'}.
		\end{align*}
		Hence
		\[
		\varphi(a,b)=\frac{1}{\omega(a,a^{-1},b)},
		\qquad
		\varphi(a,ab)=\frac{1}{\omega(a^{-1},a,b)}.
		\]
		Therefore,
		\[
		\varphi(a,ab)=\frac{1}{\omega(a,a^{-1},ab)}=\frac{1}{\omega(a^{-1},a,b)},
		\]
		so that $\omega(a,a^{-1},ab)=\omega(a^{-1},a,b)$ must hold. Since $\omega$ is a normalized $3$-cocycle, we have
		\[
		\omega(a,a^{-1},a)\omega(a^{-1},a,b)=\omega(a,a^{-1},ab).
		\]
		Thus, if $V$ is to be the inverse of $W$, it is necessary that $\omega(a,a^{-1},a)=1$ for every $a\in G$. In particular, this shows that for the first Heisenberg double $\HD_1(H^\ast)$ of a general finite-dimensional quasi-Hopf algebra $H$, the canonical element $W$ may fail to be invertible.
		
		For instance, let $G=\mathbb{Z}/n\mathbb{Z}= \{\overline{0}, \overline{1}, \dots, \overline{n-1}\}$ ($n \ge 2$), written additively, be the cyclic group of order $n$, and take $k=\mathbb{C}$. With the $3$-cocycle $\omega(\overline{a}, \overline{b}, \overline{c}) = \exp\left( \frac{2\pi \sqrt{-1}}{n} abc \right)$, we have $\omega(\overline{1},\overline{n-1},\overline{1}) = \exp\left( \frac{2\pi \sqrt{-1}(n-1)}{n}\right) \neq 1$. This provides a concrete example where $W$ is not invertible.
		
		By a completely analogous argument, if $\bar{W}$ were invertible with inverse $\bar{V}$, then there would exist a function $\psi\co G\times G \to k$ such that
		\[
		\bar{V}=\sum_{a,b \in G}\psi(a,b)\delta_a \# 1_G \otimes \delta_b \# a^{-1}.
		\]
		Comparing $\bar{W}\bar{V}$ and $\bar{V}\bar{W}$, we obtain
		\[
		\psi(a,ba)=\frac{1}{\omega(b,a,a^{-1})},
		\qquad
		\psi(a,b)=\frac{1}{\omega(b,a^{-1},a)}.
		\]
		Hence
		\[
		\omega(ba,a^{-1},a)=\omega(b,a,a^{-1}),
		\]
		and since $\omega$ is a normalized $3$-cocycle, this again implies that $\omega(a,a^{-1},a)=1$ for every $a\in G$. Therefore, the canonical element $\bar{W}$ of $\HD_1(H)$ may also fail to be invertible. The cyclic example above provides such an example as well.
	\end{rem}
	
	\begin{acknowledgment}
		The author is grateful to Sakie Suzuki, Kenichi Shimizu, and Michihisa Wakui for valuable advice and helpful discussions.
	\end{acknowledgment}
	

\end{document}